\documentclass[letterpaper, 10 pt, conference]{ieeeconf}  

\IEEEoverridecommandlockouts                              

\usepackage[dvipdfmx]{graphicx}

\usepackage{comment} 
\usepackage{color} 
\usepackage{bm} 

\usepackage{amsmath,amssymb}
\usepackage{algorithm,algorithmic}
\usepackage{amsfonts}
\usepackage{physics}
\allowdisplaybreaks[4]
\usepackage{braket}
\usepackage{url}
\usepackage{lmodern}

\newcommand{\mcal}[1]{\mathcal{#1}}
\newcommand{\mb}[1]{\mathbb{#1}}

\newcommand{\pl}{\partial}

\newtheorem{theo}{Theorem}

\newcommand{\argmin}{\mathop{\rm argmin}\limits}

\title{\LARGE \bf
Mean-Field Control Approach to Decentralized Stochastic Control with Finite-Dimensional Memories
}

\author{Takehiro Tottori$^{1}$ and Tetsuya J. Kobayashi$^{2}$
\thanks{*The first author received a JSPS Research Fellowship (Grant No. 21J20436). This work was supported by JSPS KAKENHI (Grant No. 19H05799) and JST CREST (Grant No. JPMJCR2011).}
\thanks{$^{1}$Takehiro Tottori is with 
	Department of Mathematical Informatics, Graduate School of Information Science and Technology, The University of Tokyo, Tokyo 113-8654, Japan
        {\tt\small takehiro\_tottori@sat.t.u-tokyo.ac.jp}}%
\thanks{$^{2}$Tetsuya J. Kobayashi is with Institute of Industrial Science, The University of Tokyo, Tokyo 153-8455, Japan
        {\tt\small tetsuya@sat.t.u-tokyo.ac.jp}}%
}

\begin{document}

\maketitle
\thispagestyle{empty}
\pagestyle{empty}

\begin{abstract}
Decentralized stochastic control (DSC) considers the optimal control problem of a multi-agent system. However, DSC cannot be solved except in the special cases because the estimation among the agents is generally intractable. In this work, we propose memory-limited DSC (ML-DSC), in which each agent compresses the observation history into the finite-dimensional memory. Because this compression simplifies the estimation among the agents, ML-DSC can be solved in more general cases based on the mean-field control theory. We demonstrate ML-DSC in the general LQG problem. Because estimation and control are not clearly separated in the general LQG problem, the Riccati equation is modified to the decentralized Riccati equation, which improves estimation as well as control. Our numerical experiment shows that the decentralized Riccati equation is superior to the conventional Riccati equation. 
\end{abstract}

\section{INTRODUCTION}
Control problems of multi-agent systems have many practical applications including real-time communication \cite{mahajan_optimal_2009}, decentralized detection \cite{nayyar_sequential_2011}, and networked control \cite{mahajan_optimal_2009-1}. 

Decentralized stochastic control (DSC) is a conventional theoretical framework that considers the optimal control problem of a multi-agent system \cite{nayyar_decentralized_2013,charalambous_centralized_2017,charalambous_centralized_2018}. 
DSC consists of a system and multiple controllers. 
Because each controller cannot completely observe the state of the system and the controls of the other controllers, it determines the control based on the noisy observation history. 

In order to obtain the optimal control, each controller needs to estimate the state of the system and the observation histories of the other controllers from its own observation history.  
Although the estimation of the state of the system can be accomplished by the sequential Bayesian filtering \cite{bensoussan_stochastic_1992,nisio_stochastic_2015}, that of the observation histories of the other controllers is generally intractable. 
As a result, the conventional DSC cannot be solved except in the special cases. 

In order to address this problem, we propose an alternative theoretical framework to DSC, which can be solved in more general cases. 
We call it memory-limited DSC (ML-DSC), in which each controller compresses the observation history into the finite-dimensional memory. 
Because this compression simplifies the estimation among the controllers, ML-DSC is more tractable than the conventional DSC. 

ML-DSC can be solved by employing the mathematical technique of the mean-field control theory \cite{bensoussan_master_2015,bensoussan_interpretation_2017,tottori_mean-field_2022}. 
We show that the optimal control function of ML-DSC is obtained by jointly solving the Fokker-Planck (FP) equation and the Hamilton-Jacobi-Bellman (HJB) equation. 
The system of HJB-FP equations also appears in the mean-field game and control \cite{bensoussan_master_2015,bensoussan_interpretation_2017}, and numerous numerical algorithms have been developed \cite{lauriere_numerical_2021}. 
Therefore, unlike the conventional DSC, ML-DSC can be solved in more general cases by using these algorithms.  

ML-DSC is the extension of the finite-state controller \cite{bernstein_bounded_2005,oliehoek_concise_2016,tottori_forward_2021} from the discrete setting to the continuous setting. 
However, it is difficult to extend the algorithms of the finite-state controller to our setting because they strongly depend on discreteness. 
We resolve this problem by using the trick of the mean-field control theory. 

ML-DSC is also the extension of memory-limited partially observable stochastic control (ML-POSC) \cite{tottori_mean-field_2022} from a single-agent system to a multi-agent system. 
The conventional POSC approach  \cite{bensoussan_stochastic_1992,nisio_stochastic_2015} cannot be extended to the conventional DSC because the estimation among the controllers is much difficult. 
In contrast, ML-POSC approach can be straightforwardly extended to ML-DSC because the compression of the observation histories into the finite-dimensional memories simplifies the estimation among the controllers. 

We demonstrate how ML-DSC works by applying it to the Linear-Quadratic-Gaussian (LQG) problem. 
The conventional DSC can be solved in the special LQG problems where the controllers have no information about the other controllers \cite{charalambous_centralized_2017,charalambous_centralized_2018}, or where the controllers have a nested structure \cite{lessard_optimal_2012,lessard_structural_2013}. 
In contrast, ML-DSC can be solved in a more general LQG problem involving a non-nested structure. 
Because estimation and control are not clearly separated in the general LQG problem, the Riccati equation for control is modified to include estimation, which is called the decentralized Riccati equation in this paper. 
We demonstrate that the decentralized Riccati equation is superior to the conventional Riccati equation in the general LQG problem. 

This paper is organized as follows: 
In Sec. \ref{sec: Review of DSC}, we briefly review the conventional DSC. 
In Sec. \ref{sec: ML-DSC}, we formulate ML-DSC. 
In Sec. \ref{sec: MFC}, we solve ML-DSC based on the mean-field control theory. 
In Sec. \ref{sec: Application to LQG problem}, we apply ML-DSC to the LQG problem.   
In Sec. \ref{sec: Conclusion}, we conclude this paper. 

\section{REVIEW OF DECENTRALIZED STOCHASTIC CONTROL}\label{sec: Review of DSC}
In this section, we briefly review the conventional DSC \cite{charalambous_centralized_2017,charalambous_centralized_2018}. 
DSC consists of a system and $N$ controllers.  
$x_{t}\in\mb{R}^{d_{x}}$ is the state of the system at time $t\in[0,T]$, which evolves by the following stochastic differential equation (SDE): 
\begin{align}
	dx_{t}&=b(t,x_{t},u_{t})dt+\sigma(t,x_{t},u_{t})d\omega_{t},\label{eq: state SDE}
\end{align}
where $x_{0}$ obeys $p_{0}(x_{0})$, $\omega_{t}\in\mb{R}^{d_{\omega}}$ is the standard Wiener process, $u_{t}^{i}\in\mb{R}^{d_{u}^{i}}$ is the control of the controller $i\in\{1,...,N\}$, and $u_{t}:=(u_{t}^{1},u_{t}^{2},...,u_{t}^{N})$ is the joint control of $N$ controllers. 

In DSC, because the controller $i$ cannot completely observe the state $x_{t}$ and the joint control $u_{t}$, 
the controller $i$ obtains the observation $y_{t}^{i}\in\mb{R}^{d_{y}^{i}}$ instead of them, which evolves by the following SDE: 
\begin{align}
	dy_{t}^{i}&=h^{i}(t,x_{t},u_{t})dt+\gamma^{i}(t,x_{t},u_{t})d\nu_{t}^{i},\label{eq: observation SDE}
\end{align}
where $y_{0}^{i}$ obeys $p_{0}^{i}(y_{0}^{i})$, and $\nu_{t}^{i}\in\mb{R}^{d_{\nu}^{i}}$ is the standard Wiener process. 

The controller $i$ determines the control $u_{t}^{i}$ based on the observation history $y_{0:t}^{i}:=\{y_{\tau}^{i}|\tau\in[0,t]\}$ as follows: 
\begin{align}
	u_{t}^{i}=u^{i}(t,y_{0:t}^{i}).\label{eq: control of DSC}
\end{align}

The objective function of DSC is given by the following expected cumulative cost function: 
\begin{align}	
	J[u]:=\mb{E}_{u}\left[\int_{0}^{T}f(t,x_{t},u_{t})dt+g(x_{T})\right],
	\label{eq: OF of DSC}
\end{align}
where $f$ is the cost function, and $g$ is the terminal cost function. 
DSC is the problem to find the optimal control function $u^{*}$ that minimizes the objective function $J[u]$ as follows: 
\begin{align}	
	u^{*}:=\argmin_{u}J[u].
	\label{eq: DSC}
\end{align}

In order to obtain the optimal control function $u^{*}$, the controller $i$ needs to estimate the state of the system $x_{t}$ and the observation histories of the other controllers $y_{0:t}^{j}\ (j\neq i)$ from its own observation history $y_{0:t}^{i}$, which is generally intractable. 
As a result, the conventional DSC cannot be solved except in the special cases. 

\section{MEMORY-LIMITED DECENTRALIZED STOCHASTIC CONTROL}\label{sec: ML-DSC}
In order to address this problem, we propose an alternative theoretical framework to the conventional DSC, ML-DSC. 
In this section, we formulate ML-DSC. 

\subsection{Problem formulation}
In this subsection, we formulate ML-DSC. 
In ML-DSC, the controller $i$ determines the control $u_{t}^{i}$ based on the finite-dimensional memory $z_{t}^{i}\in\mb{R}^{d_{z}^{i}}$ as follows: 
\begin{align}
	u_{t}^{i}=u^{i}(t,z_{t}^{i}).\label{eq: control of ML-DSC}
\end{align}
$d_{z}^{i}$ is determined by the dimension of the memory available to the controller $i$. 
Comparing (\ref{eq: control of DSC}) and (\ref{eq: control of ML-DSC}), the memory $z_{t}^{i}$ can be interpreted as the compression of the observation history $y_{0:t}^{i}$. 
Because this compression simplifies the estimation among the controllers, ML-DSC is more tractable than the conventional DSC.

The memory $z_{t}^{i}$ is assumed to evolve by
\begin{align}
	dz_{t}^{i}=c^{i}(t,z_{t}^{i},v_{t}^{i})dt+\kappa^{i}(t,z_{t}^{i},v_{t}^{i})dy_{t}^{i}, 
	\label{eq: memory SDE}
\end{align}
where $z_{0}^{i}$ obeys $p_{0}^{i}(z_{0}^{i})$, and $v_{t}^{i}=v^{i}(t,z_{t}^{i})\in\mb{R}^{d_{v}^{i}}$ is the control. 
Because (\ref{eq: memory SDE}) depends on the observation $dy_{t}^{i}$, the observation history $y_{0:t}^{i}$ can be compressed into the memory $z_{t}^{i}$. 
Furthermore, because (\ref{eq: memory SDE}) depends on the control $v_{t}^{i}$, the memory $z_{t}^{i}$ can be optimized through the control $v_{t}^{i}$, which can improve the estimation. 
We note that (\ref{eq: memory SDE}) can be extended to include the intrinsic stochasticity \cite{tottori_mean-field_2022}. 

The objective function of ML-DSC is given by the following expected cumulative cost function:  
\begin{align}
	J[u,v]:=\mb{E}_{u,v}\left[\int_{0}^{T}f(t,x_{t},u_{t},v_{t})dt+g(x_{T})\right]. 
	\label{eq: OF of ML-DSC}
\end{align}
Because the cost function $f$ depends on the memory control $v_{t}$ as well as the state control $u_{t}$, ML-DSC can consider the memory control cost (estimation cost) as well as the state control cost (control cost) \cite{tottori_mean-field_2022}. 
In the light of the dualistic roles played by estimation and control, it is natural to consider the estimation cost as well as the control cost. 

ML-DSC optimizes the state control function $u$ and the memory control function $v$ based on the objective function $J[u,v]$ as follows: 
\begin{align}
	u^{*},v^{*}:=\argmin_{u,v}J[u,v].
	\label{eq: ML-DSC}
\end{align}

\subsection{Problem reformulation}
Although the formulation of ML-DSC in the previous subsection clarifies its  relationship with the conventional DSC, it is inconvenient for further mathematical investigations. 
In order to resolve this problem, we reformulate ML-DSC in this subsection. 
This formulation is simpler and more general than the previous one. 

We first define the extended state $s_{t}$ as follows: 
\begin{align}
	s_{t}:=\left(\begin{array}{c}
		x_{t}\\ 
		z_{t}^{1}\\
		\vdots\\
		z_{t}^{N}\\
	\end{array}\right)\in\mb{R}^{d_{s}}, 
\end{align}
where $d_{s}=d_{x}+\sum_{i=1}^{N}d_{z}^{i}$. 
The extended state $s_{t}$ evolves by the following SDE: 
\begin{align}
	ds_{t}=\tilde{b}(t,s_{t},\tilde{u}_{t})dt+\tilde{\sigma}(t,s_{t},\tilde{u}_{t})d\tilde{\omega}_{t}, 
	\label{eq: extended state SDE}
\end{align}
where $s_{0}$ obeys $p_{0}(s_{0})$, $\tilde{\omega}_{t}\in\mb{R}^{d_{\tilde{\omega}}}$ is the standard Wiener process, $\tilde{u}_{t}^{i}\in\mb{R}^{d_{\tilde{u}}^{i}}$ is the control of the controller $i$, and $\tilde{u}_{t}:=(\tilde{u}_{t}^{1},\tilde{u}_{t}^{2},...,\tilde{u}_{t}^{N})$ is the joint control of $N$ controllers. 
In ML-DSC, the controller $i$ determines the control $\tilde{u}_{t}^{i}$ based on the memory $z_{t}^{i}$ as follows: 
\begin{align}
	\tilde{u}_{t}^{i}=\tilde{u}^{i}(t,z_{t}^{i}).\label{eq: control of GML-DSC}
\end{align}
The extended state SDE (\ref{eq: extended state SDE}) includes the previous state, observation, and memory SDEs (\ref{eq: state SDE}), (\ref{eq: observation SDE}), (\ref{eq: memory SDE}) as a special case because they can be represented as follows: 
\begin{align}
	&ds_{t}=\left(\begin{array}{c}
		b\\
		c^{1}+\kappa^{1}h^{1}\\
		\vdots\\
		c^{N}+\kappa^{N}h^{N}\\
	\end{array}\right)dt\nonumber\\
	&+\left(\begin{array}{cccc}
		\sigma&O&\cdots&O\\
		O&\kappa^{1}\gamma^{1}&\cdots&O\\
		\vdots&\vdots&\ddots&\vdots\\
		O&O&\cdots&\kappa^{N}\gamma^{N}\\
	\end{array}\right)
	\left(\begin{array}{c}
		d\omega_{t}\\
		d\nu_{t}^{1}\\
		\vdots\\
		d\nu_{t}^{N}\\
	\end{array}\right),
\end{align}
where $p_{0}(s_{0})=p_{0}(x_{0})p_{0}(z_{0})$. 

The objective function of ML-DSC is given by the following expected cumulative cost function: 
\begin{align}	
	J[\tilde{u}]:=\mb{E}_{\tilde{u}}\left[\int_{0}^{T}\tilde{f}(t,s_{t},\tilde{u}_{t})dt+\tilde{g}(s_{T})\right],
	\label{eq: OF of GML-DSC}
\end{align}
where $\tilde{f}$ is the cost function, and $\tilde{g}$ is the terminal cost function.  
It is obvious that this objective function (\ref{eq: OF of GML-DSC}) is more general than that in the previous one (\ref{eq: OF of ML-DSC}). 

ML-DSC is the problem to find the optimal control function $\tilde{u}^{*}$ that minimizes the objective function $J[\tilde{u}]$ as follows: 
\begin{align}	
	\tilde{u}^{*}:=\argmin_{u}J[\tilde{u}].
	\label{eq: GML-DSC}
\end{align}

In the following section, we mainly consider the formulation of this subsection rather than that of the previous subsection because it is simpler and more general. 
Moreover, we omit $\tilde{\cdot}$ for the notational simplicity.

\section{MEAN-FIELD CONTROL APPROACH}\label{sec: MFC}
If the control $u_{t}^{i}$ is determined based on the extended state $s_{t}$, i.e., $u_{t}^{i}=u^{i}(t,s_{t})$, ML-DSC is the same with the completely observable stochastic control (COSC) of the extended state, and it can be solved by the conventional COSC approach \cite{yong_stochastic_1999}. 
However, because ML-DSC determines the control $u_{t}^{i}$ based solely on the memory $z_{t}^{i}$, i.e., $u_{t}^{i}=u^{i}(t,z_{t}^{i})$, ML-DSC cannot be approached in the similar way as COSC. 
In this section, we propose the mean-field control approach  \cite{tottori_mean-field_2022} to ML-DSC. 

\subsection{Derivation of optimal control function}
In this subsection, we solve ML-DSC based on the mean-field control theory \cite{tottori_mean-field_2022}. 
We first show that ML-DSC can be converted into a deterministic control of the probability density function.
The extended state SDE (\ref{eq: extended state SDE}) can be converted into the following Fokker-Planck (FP) equation: 
\begin{align}
	\frac{\pl p_{t}(s)}{\pl t}=\mcal{L}^{\dagger}p_{t}(s), 
	\label{eq: FP eq}
\end{align}
where the initial condition is given by $p_{0}(s)$, and $\mcal{L}^{\dag}$ is the forward diffusion operator, which is defined by
\begin{align}
	\mcal{L}^{\dag}p(s)&:=-\sum_{i=1}^{d_{s}}\frac{\pl (b_{i}(t,s,u)p(s))}{\pl s_{i}}\nonumber\\
	&+\frac{1}{2}\sum_{i,j=1}^{d_{s}}\frac{\pl^{2}(D_{ij}(t,s,u)p(s))}{\pl s_{i}\pl s_{j}}, \nonumber
\end{align}
where $D(t,s,u):=\sigma(t,s,u)\sigma^{T}(t,s,u)$. 
The objective function of ML-DSC (\ref{eq: OF of GML-DSC}) can be calculated as follows: 
\begin{align}
	J[u]=\int_{0}^{T}\bar{f}(t,p_{t},u_{t})dt+\bar{g}(p_{T}), 
	\label{eq: OF of GML-DSC FP}
\end{align}
where $\bar{f}(t,p,u):=\mb{E}_{p(s)}[f(t,s,u)]$ and $\bar{g}(p):=\mb{E}_{p(s)}[g(s)]$. 
From (\ref{eq: FP eq}) and (\ref{eq: OF of GML-DSC FP}), ML-DSC is converted into a deterministic control of $p_{t}$. 
As a result, ML-DSC can be approached in the similar way as the deterministic control. 

\begin{theo}
\label{theo: optimal control of GML-DSC based on Bellman eq}
The optimal control function of ML-DSC is given by 
\begin{align}
	&u^{i*}(t,z^{i})=\argmin_{u^{i}}\nonumber\\
	&\mb{E}_{p_{t}(s^{-i}|z^{i})}\left[H\left(t,s,(u^{-i*},u^{i}),\frac{\delta V(t,p_{t})}{\delta p}(s)\right)\right],
	\label{eq: optimal control of GML-DSC based on Bellman eq}
\end{align}
where $s^{-i}$ and $(u^{-i*},u^{i})$ are defined by 
\begin{align}
	s^{-i}&:=(x,z^{1},...,z^{i-1},z^{i+1},...,z^{N}),\nonumber\\
	(u^{-i*},u^{i})&:=(u^{1*},...,u^{i-1*},u^{i},u^{i+1*},...,u^{N*}), \nonumber
\end{align}
and $H$ is the Hamiltonian, which is defined by
\begin{align}
	H\left(t,s,u^{*},\frac{\delta V(t,p)}{\delta p}(s)\right):=f(t,s,u)+\mcal{L}\frac{\delta V(t,p)}{\delta p}(s),\nonumber
\end{align}
where $\mcal{L}$ is the backward diffusion operator, which is defined by
\begin{align}
	\mcal{L}p(s)&:=\sum_{i=1}^{d_{s}}b_{i}(t,s,u)\frac{\pl p(s)}{\pl s_{i}}
	+\frac{1}{2}\sum_{i,j=1}^{d_{s}}D_{ij}(t,s,u)\frac{\pl^{2} p(s)}{\pl s_{i}\pl s_{j}}.\nonumber
\end{align}
We note that $\mcal{L}$ is the conjugate of $\mcal{L}^{\dag}$. 
$p_{t}(s^{-i}|z^{i})=p_{t}(s)/\int p_{t}(s)ds^{-i}$, $p_{t}(s)$ is the solution of the FP equation (\ref{eq: FP eq}), and $V(t,p)$ is the solution of the following Bellman equation: 
\begin{align}
	-\frac{\pl V(t,p)}{\pl t}=\mb{E}_{p(s)}\left[H\left(t,s,u^{*},\frac{\delta V(t,p)}{\delta p}(s)\right)\right], 
	\label{eq: Bellman eq of GML-DSC}
\end{align}
where $V(T,p)=\mb{E}_{p(s)}[g(s)]$. 
\end{theo}

\begin{proof}
The proof is shown in Appendix A. 
\end{proof}

However, because the Bellman equation (\ref{eq: Bellman eq of GML-DSC}) is a functional differential equation, it cannot be solved even numerically. 
We resolve this problem by employing the mathematical technique of the mean-field control theory \cite{bensoussan_master_2015,bensoussan_interpretation_2017,tottori_mean-field_2022}. 
This technique converts Theorem \ref{theo: optimal control of GML-DSC based on Bellman eq} into the following theorem by defining 
\begin{align}
	w(t,s):=\frac{\delta V(t,p_{t})}{\delta p}(s), 
\end{align}
where $p_{t}$ is the solution of the FP equation (\ref{eq: FP eq}). 

\begin{theo}\label{theo: optimal control of GML-DSC}
The optimal control function of ML-DSC is given by
\begin{align}
	&u^{i*}(t,z^{i})=\argmin_{u^{i}}\nonumber\\
	&\mb{E}_{p_{t}(s^{-i}|z^{i})}\left[H\left(t,s,(u^{-i*},u^{i}),w\right)\right],
	\label{eq: optimal control of GML-DSC}
\end{align}
where $p_{t}(s^{-i}|z^{i})=p_{t}(s)/\int p_{t}(s)ds^{-i}$, $p_{t}(s)$ is the solution of the FP equation (\ref{eq: FP eq}), and $w(t,s)$ is the solution of the following Hamilton-Jacobi-Bellman (HJB) equation: 
\begin{align}
	-\frac{\pl w(t,s)}{\pl t}=H\left(t,s,u^{*},w\right),
	\label{eq: HJB eq}
\end{align}
where $w(T,s)=g(s)$. 
\end{theo}

\begin{proof}
The proof is almost the same with \cite{tottori_mean-field_2022}. 
\end{proof}

While the Bellman equation (\ref{eq: Bellman eq of GML-DSC}) is a functional differential equation, the HJB equation (\ref{eq: HJB eq}) is a partial differential equation, which can be solved numerically. 

The optimal control function of ML-DSC (\ref{eq: optimal control of GML-DSC}) is obtained by jointly solving the FP equation (\ref{eq: FP eq}) and the HJB equation (\ref{eq: HJB eq}). 
The system of HJB-FP equations also appears in the mean-field game and control \cite{bensoussan_master_2015,bensoussan_interpretation_2017}, and numerous numerical algorithms have been developed \cite{lauriere_numerical_2021}. 
As a result, unlike the conventional DSC, ML-DSC can be solved in more general cases by using these algorithms.  

One of the most basic algorithms is the forward-backward sweep method (fixed-point iteration method) \cite{lauriere_numerical_2021,tottori_notitle_2022}, which computes the FP equation (\ref{eq: FP eq}) and the HJB equation (\ref{eq: HJB eq}) alternately. 
While the convergence of the forward-backward sweep method is not guaranteed in the mean-field game and control, it is guaranteed in ML-DSC because the coupling of HJB-FP equations is limited to the optimal control function in ML-DSC \cite{tottori_notitle_2022}.

\subsection{Comparison with completely observable or memory-limited partially observable stochastic control}
The COSC of the extended state and ML-POSC can be solved in the similar way as ML-DSC \cite{tottori_mean-field_2022}. 

In the COSC of the extended state, because the control $u_{t}^{i}$ is determined based on the extended state $s_{t}$, i.e., $u_{t}^{i}=u^{i}(t,s_{t})$, the optimal control function is given by 
\begin{align}
	u^{i*}(t,s)=\argmin_{u^{i}}H\left(t,s,(u^{-i*},u^{i}),w\right).
	\label{eq: optimal control of COSC}
\end{align}

In ML-POSC, because the control $u_{t}^{i}$ is determined based on the joint memory $z_{t}$, i.e., $u_{t}^{i}=u^{i}(t,z_{t})$, the optimal control function is given by 
\begin{align}
	u^{i*}(t,z)=\argmin_{u^{i}}\mb{E}_{p_{t}(x|z)}\left[H\left(t,s,(u^{-i*},u^{i}),w\right)\right]. 
	\label{eq: optimal control of ML-POSC}
\end{align}

Although the HJB equation (\ref{eq: HJB eq}) is the same between COSC, ML-POSC, and ML-DSC, the optimal control function is different. 
Especially, the optimal control functions of ML-POSC and ML-DSC depend on the FP equation (\ref{eq: FP eq}) because they need to estimate unobservables from observables. 

\section{LINEAR-QUADRATIC-GAUSSIAN PROBLEM}\label{sec: Application to LQG problem}
In this section, we demonstrate how ML-DSC works by applying it to the general LQG problem involving a non-nested structure. 

\subsection{Problem formulation}
In this subsection, we formulate the LQG problem \cite{bensoussan_estimation_2018}. 
The extended state SDE (\ref{eq: extended state SDE}) is given as follows:  
\begin{align}
	ds_{t}&=\left(A(t)s_{t}+B(t)u_{t}\right)dt+\sigma(t)d\omega_{t}\nonumber\\
	&=\left(A(t)s_{t}+\sum_{i=1}^{N}B_{i}(t)u_{t}^{i}\right)dt+\sigma(t)d\omega_{t},\label{SDE of LQG}
\end{align}
where the initial condition is given by the Gaussian distribution $p_{0}(s):=\mcal{N}\left(s\left|\mu_{0},\Sigma_{0}\right.\right)$. 
The objective function (\ref{eq: OF of GML-DSC}) is given as follows:  
\begin{align}
	&J[u]:=\mb{E}_{u}\left[\int_{0}^{T}\left(s_{t}^{T}Qs_{t}+u_{t}^{T}Ru_{t}\right)dt+s_{T}^{T}Ps_{T}\right],\label{OF of LQG}
\end{align}
where $Q(t)\succeq O$, $R(t)\succ O$, and $P\succeq O$. 
The objective of this problem is to find the optimal control function $u^{*}$ that minimizes the objective function $J[u]$. 

In this paper, we assume that $R(t)$ is the block diagonal matrix as follows: 
\begin{align}
	R(t)=\left(\begin{array}{cccc}
		R_{11}(t)&O&\cdots&O\\
		O&R_{22}(t)&\cdots&O\\
		\vdots&\vdots&\ddots&\vdots\\
		O&O&\cdots&R_{NN}(t)\\
	\end{array}\right),
	\label{block diagonal assumption}
\end{align}
where $R_{ii}(t)\succ O\in\mb{R}^{d_{u}^{i}\times d_{u}^{i}}$. 
If this assumption does not hold, the optimal control function cannot be derived explicitly. 
This problem is similar with the Witsenhausen's counterexample \cite{witsenhausen_counterexample_1968}. 

\subsection{Derivation of optimal control function}\label{sec: Decentralized Riccati equation}
In this subsection, we derive the optimal control function of the LQG problem. 
In the LQG problem, the probability density function of the extended state $s$ at time $t$ is given by the Gaussian distribution $p_{t}(s):=\mcal{N}\left(s|\mu(t),\Sigma(t)\right)$. 
Defining the stochastic extended state $\hat{s}:=s-\mu$, $\mb{E}_{p_{t}(s^{-i}|z^{i})}\left[s\right]$ is given as follows: 
\begin{align}
	\mb{E}_{p_{t}(s^{-i}|z^{i})}\left[s\right]=K_{i}(t)\hat{s}+\mu(t), 
	\label{eq: conditional mean vector of LQG}
\end{align}
where $K_{i}(t)$ is defined by 
\begin{align}
	K_{i}(t):=\left(\begin{array}{ccccc}
		O&\cdots&\Sigma_{xz^{i}}(t)\Sigma_{z^{i}z^{i}}^{-1}(t)&\cdots&O\\
		O&\cdots&\Sigma_{z^{1}z^{i}}(t)\Sigma_{z^{i}z^{i}}^{-1}(t)&\cdots&O\\
		\vdots&\ddots&\vdots&\ddots&\vdots\\
		O&\cdots&I&\cdots&O\\
		\vdots&\ddots&\vdots&\ddots&\vdots\\
		O&\cdots&\Sigma_{z^{N}z^{i}}(t)\Sigma_{z^{i}z^{i}}^{-1}(t)&\cdots&O\\
	\end{array}\right).\nonumber
\end{align}
$K_{i}(t)$ is the zero matrix except for the columns corresponding to $z^{i}$. 
By applying Theorem \ref{theo: optimal control of GML-DSC} to the LQG problem, we obtain the following theorem: 

\begin{theo}\label{theo: optimal control of LQG in ML-DSC}
In the LQG problem of ML-DSC, the optimal control function is given by
\begin{align}
	u^{i*}(t,z^{i})=-R_{ii}^{-1}B_{i}^{T}\left(\Phi K_{i}\hat{s}+\Psi\mu\right).
	\label{eq: optimal control of LQG}
\end{align}
where $K_{i}(t)$ depends on $\Sigma(t)$, and $\mu(t)$ and $\Sigma(t)$ are the solutions of the following ordinary differential equations: 
\begin{align}
	\dot{\mu}&=\left(A-BR^{-1}B^{T}\Psi\right)\mu\label{eq: ODE of mu},\\
	\dot{\Sigma}&=\sigma\sigma^{T}+\left(A-\sum_{i=1}^{N}B_{i}R_{ii}^{-1}B_{i}^{T}\Phi K_{i}\right)\Sigma\nonumber\\
	&\ \ \ \ \ \ \ \ \ +\Sigma\left(A-\sum_{i=1}^{N}B_{i}R_{ii}^{-1}B_{i}^{T}\Phi K_{i}\right)^{T},\label{eq: ODE of Sigma}
\end{align}
where $\mu(0)=\mu_{0}$ and $\Sigma(0)=\Sigma_{0}$. 
$\Psi(t)$ and $\Phi(t)$ are the solutions of the following ordinary differential equations: 
\begin{align}
	-\dot{\Psi}&=Q+A^{T}\Psi+\Psi A -\Psi BR^{-1}B^{T}\Psi,\label{eq: ODE of Psi}\\
	-\dot{\Phi}&=Q+A^{T}\Phi+\Phi A-\Phi BR^{-1}B^{T}\Phi\nonumber\\
	&+\sum_{i=1}^{N}(I-K_{i})^{T}\Phi B_{i}R_{ii}^{-1}B_{i}^{T}\Phi (I-K_{i}), \label{eq: ODE of Pi}
\end{align}
where $\Psi(T)=\Phi(T)=P$. 
\end{theo}

\begin{proof}
The proof is shown in Appendix B.
\end{proof}

While (\ref{eq: ODE of Psi}) is the Riccati equation \cite{bensoussan_estimation_2018,charalambous_centralized_2017,charalambous_centralized_2018}, (\ref{eq: ODE of Pi}) is a new equation of ML-DSC, which is the called the decentralized Riccati equation in this paper. 
Because estimation and control are not clearly separated in the general LQG problem \cite{tottori_mean-field_2022,lessard_optimal_2012,lessard_structural_2013}, the Riccati equation (\ref{eq: ODE of Psi}) for control is modified to include estimation, which corresponds to the decentralized Riccati equation (\ref{eq: ODE of Pi}). 
As a result, the decentralized Riccati equation (\ref{eq: ODE of Pi}) may improve estimation as well as control. 

In order to support this interpretation, we analyze the decentralized Riccati equation (\ref{eq: ODE of Pi}) by comparing it with the Riccati equation (\ref{eq: ODE of Psi}). 
Since only the last term of (\ref{eq: ODE of Pi}) is different from (\ref{eq: ODE of Psi}), we denote it as follows: 
\begin{align}
	\mcal{Q}_{i}:=(I-K_{i})^{T}\Phi B_{i}R_{ii}^{-1}B_{i}^{T}\Phi (I-K_{i}). 
\end{align}
We focus on $\mcal{Q}_{N}$ for the sake of simplicity. 
Similar discussions are possible for $\mcal{Q}_{i}\ (i\in\{1,...,N-1\})$. 
We also denote $a:=s^{-N}$ and $b:=z^{N}$ for the notational simplicity. 
$a$ is unobservable and $b$ is observable for the controller $N$. 
$\mcal{Q}_{N}$ can be calculated as follows: 
\begin{align}
	\mcal{Q}_{N}=
	\left(\begin{array}{cc}
		\mcal{P}_{aa}&-\mcal{P}_{aa}\Sigma_{ab}\Sigma_{bb}^{-1}\\ 
		-\Sigma_{bb}^{-1}\Sigma_{ba}\mcal{P}_{aa}&\Sigma_{bb}^{-1}\Sigma_{ba}\mcal{P}_{aa}\Sigma_{ab}\Sigma_{bb}^{-1}\\
	\end{array}\right),
\end{align}
where $\mcal{P}_{aa}:=(\Phi B_{N}R_{NN}^{-1}B_{N}^{T}\Phi)_{aa}$. 
Because $\mcal{P}_{aa}\succeq O$ and $\Sigma_{bb}^{-1}\Sigma_{ba}\mcal{P}_{aa}\Sigma_{ab}\Sigma_{bb}^{-1}\succeq O$, 
$\Phi_{aa}$ and $\Phi_{bb}$ may be larger than $\Psi_{aa}$ and $\Psi_{bb}$, respectively. 
Because $\Phi_{aa}$ and $\Phi_{bb}$ are the negative feedback gains of $a$ and $b$, respectively, $\mcal{Q}_{N}$ may decrease $\Sigma_{aa}$ and $\Sigma_{bb}$. 
Moreover, when $\Sigma_{ab}$ is positive/negative, $\Phi_{ab}$ may be smaller/larger than $\Psi_{ab}$, which may increase/decrease $\Sigma_{ab}$. 
The similar discussion is possible for $\Sigma_{ba}$, $\Phi_{ba}$, and $\Psi_{ba}$ because $\Sigma$, $\Phi$, and $\Psi$ are symmetric matrices. 
As a result, $\mcal{Q}_{N}$ may decrease the following conditional covariance matrix: 
\begin{align}
	\Sigma_{a|b}:=\Sigma_{aa}-\Sigma_{ab}\Sigma_{bb}^{-1}\Sigma_{ba}, \label{eq: estimation error}
\end{align}
which corresponds to the estimation error of $a$ from $b$. 
Therefore, the decentralized Riccati equation (\ref{eq: ODE of Pi}) may improve estimation as well as control.

\subsection{Comparison with completely observable or memory-limited partially observable stochastic control}
In the COSC of the extended state, the optimal control function is given as follows \cite{bensoussan_estimation_2018}: 
\begin{align}
	u^{i*}(t,s)=-R_{ii}^{-1}B_{i}^{T}\left(\Psi \hat{s}+\Psi\mu\right), 
	\label{eq: optimal control of LQG COSC}
\end{align}
where $\Psi(t)$ is the solution of the Riccati equation (\ref{eq: ODE of Psi}). 

In ML-POSC, the optimal control function is given as follows \cite{tottori_mean-field_2022}: 
\begin{align}
	u^{i*}(t,z)=-R_{ii}^{-1}B_{i}^{T}\left(\Pi K\hat{s}+\Psi\mu\right),
	\label{eq: optimal control of LQG ML-POSC}
\end{align}
where $\Pi(t)$ is the solution of the partially observable Riccati equation, which is given by 
\begin{align}
	-\dot{\Pi}&=Q+A^{T}\Pi+\Pi A-\Pi BR^{-1}B^{T}\Pi\nonumber\\
	&+(I-K)^{T}\Pi BR^{-1}B^{T}\Pi (I-K), \label{eq: ODE of Phi}
\end{align}
where $\Pi(T)=P$ and $K(t)$ is defined by 
\begin{align}
	K(t):=\left(\begin{array}{cc}
		O&\Sigma_{xz}(t)\Sigma_{zz}^{-1}(t)\\
		O&I\\
	\end{array}\right).
\end{align}
The decentralized Riccati equation (\ref{eq: ODE of Pi}) is a natural extension of the partially observable Riccati equation (\ref{eq: ODE of Phi}) from a single-agent system to a multi-agent system.

\subsection{Numerical experiment}
\begin{figure*}[t]
\begin{center}
	\begin{minipage}[t][][b]{42mm}
	(a)
	\end{minipage}
	\begin{minipage}[t][][b]{42mm}
	(b)
	\end{minipage}
	\begin{minipage}[t][][b]{42mm}
	(c)
	\end{minipage}\\
	\begin{minipage}[t][][b]{42mm}
		\includegraphics[width=42mm]{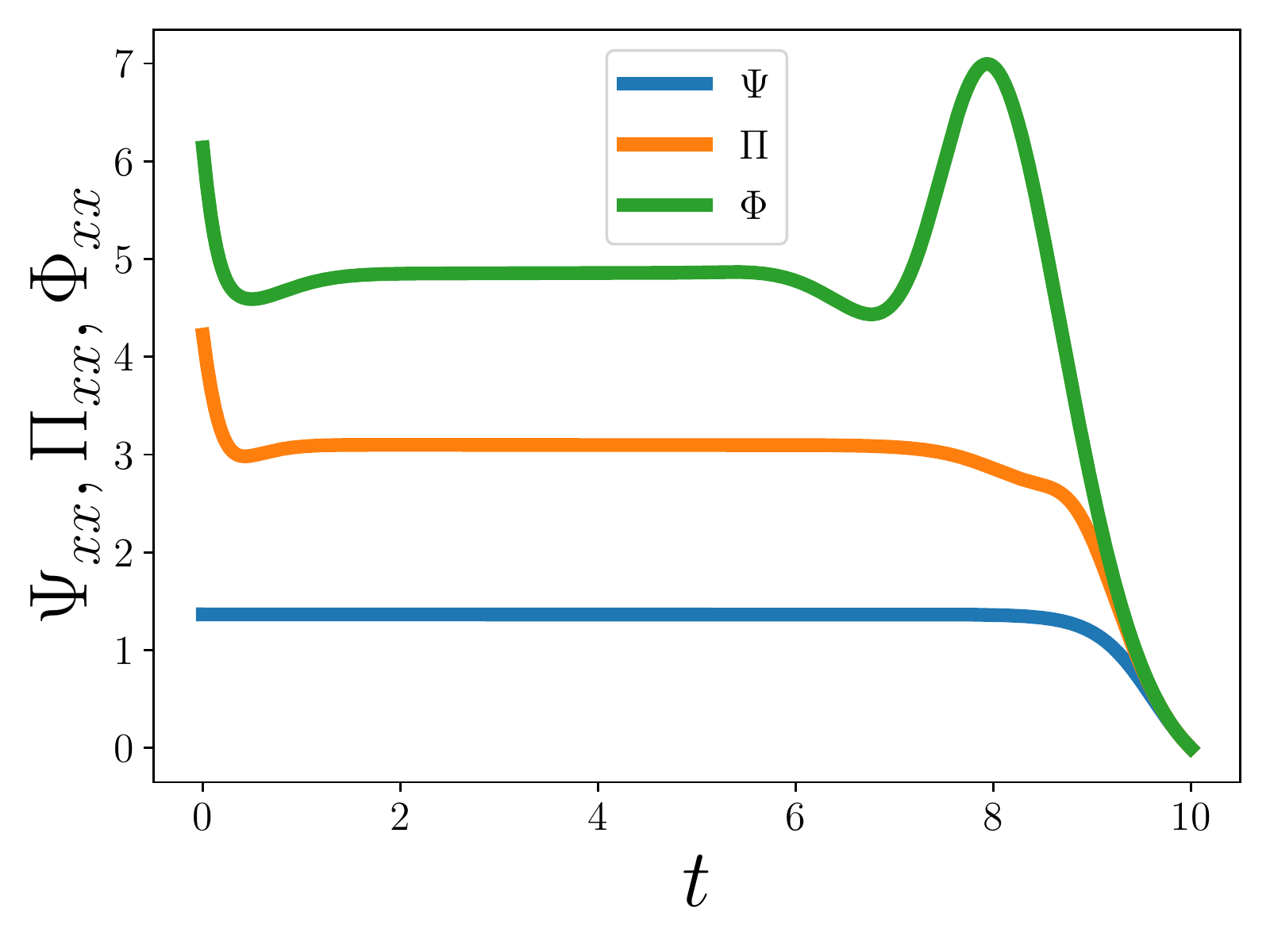}
	\end{minipage}
	\begin{minipage}[t][][b]{42mm}
		\includegraphics[width=42mm]{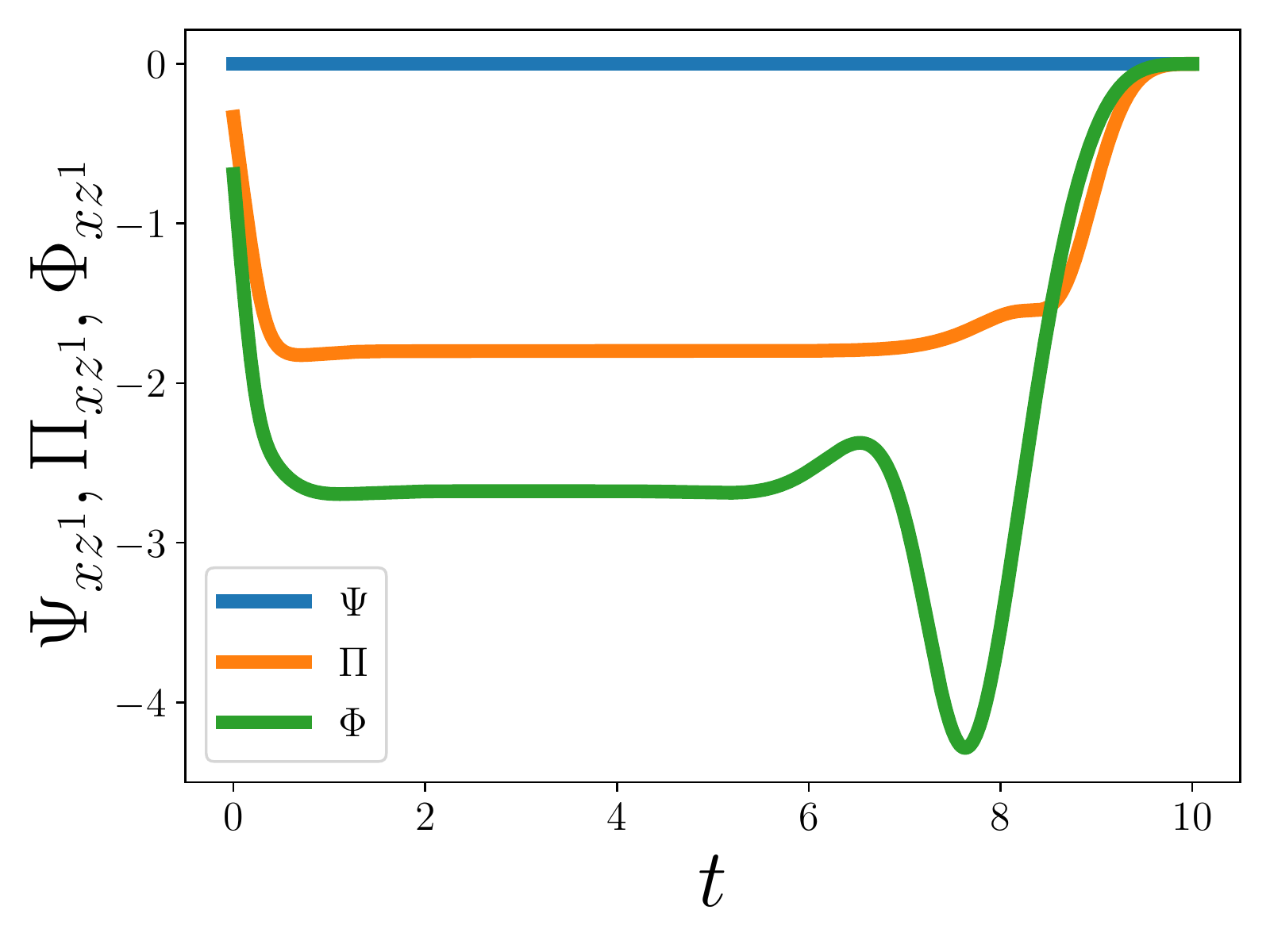}
	\end{minipage}
	\begin{minipage}[t][][b]{42mm}
		\includegraphics[width=42mm]{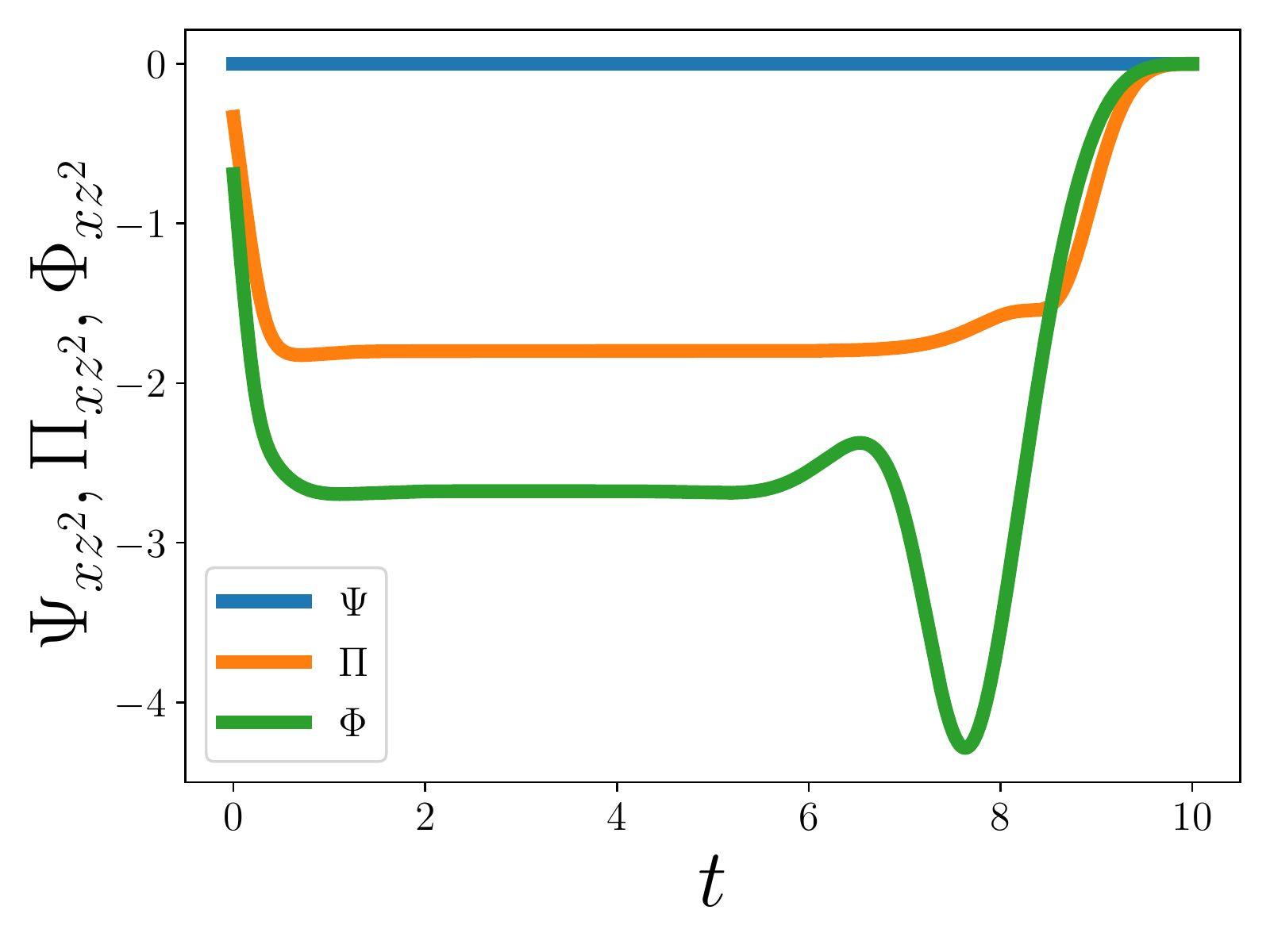}
	\end{minipage}\\
	\begin{minipage}[t][][b]{42mm}
	(d)
	\end{minipage}
	\begin{minipage}[t][][b]{42mm}
	(e)
	\end{minipage}
	\begin{minipage}[t][][b]{42mm}
	(f)
	\end{minipage}\\
	\begin{minipage}[t][][b]{42mm}
		\includegraphics[width=42mm]{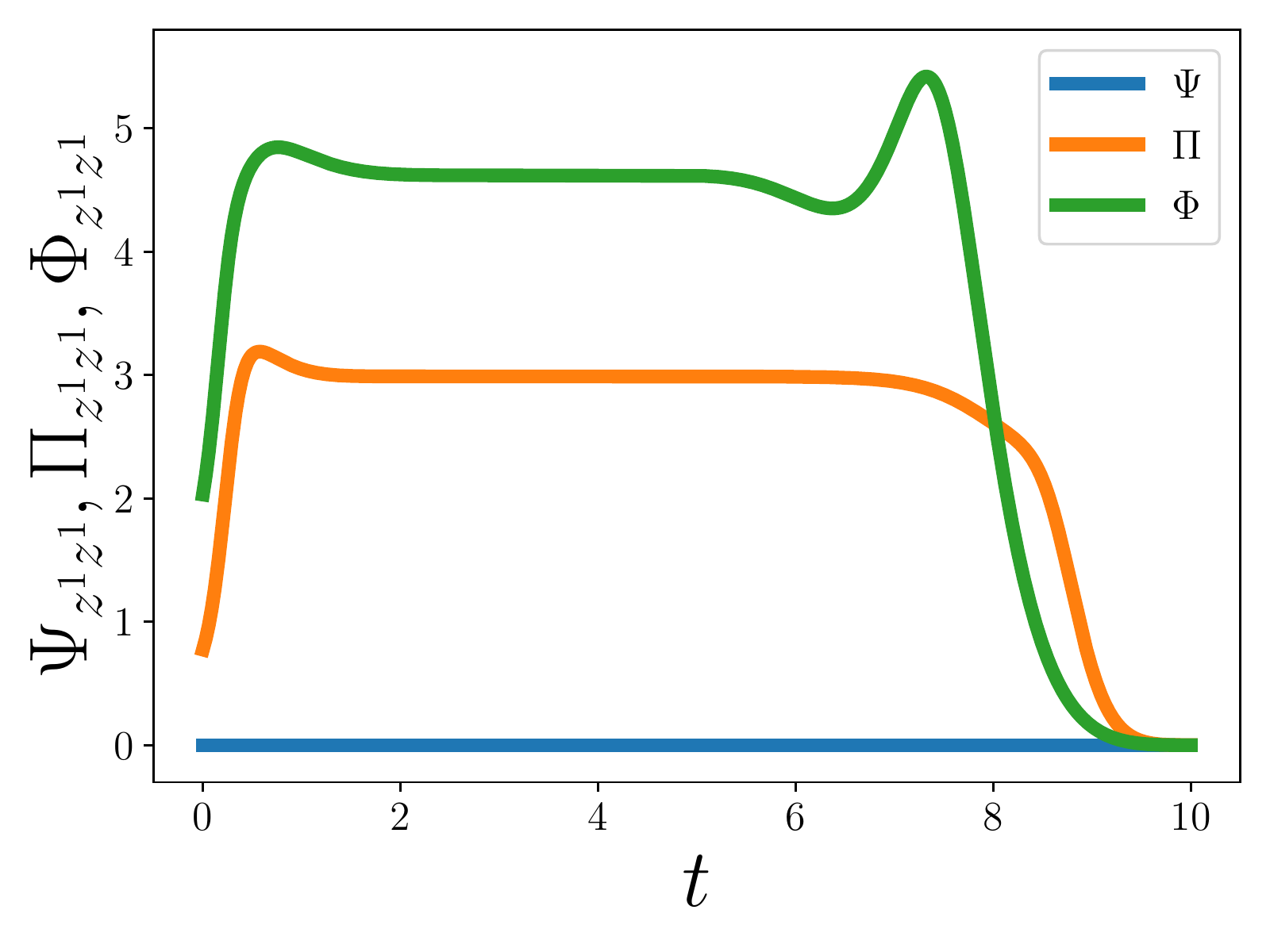}
	\end{minipage}
	\begin{minipage}[t][][b]{42mm}
		\includegraphics[width=42mm]{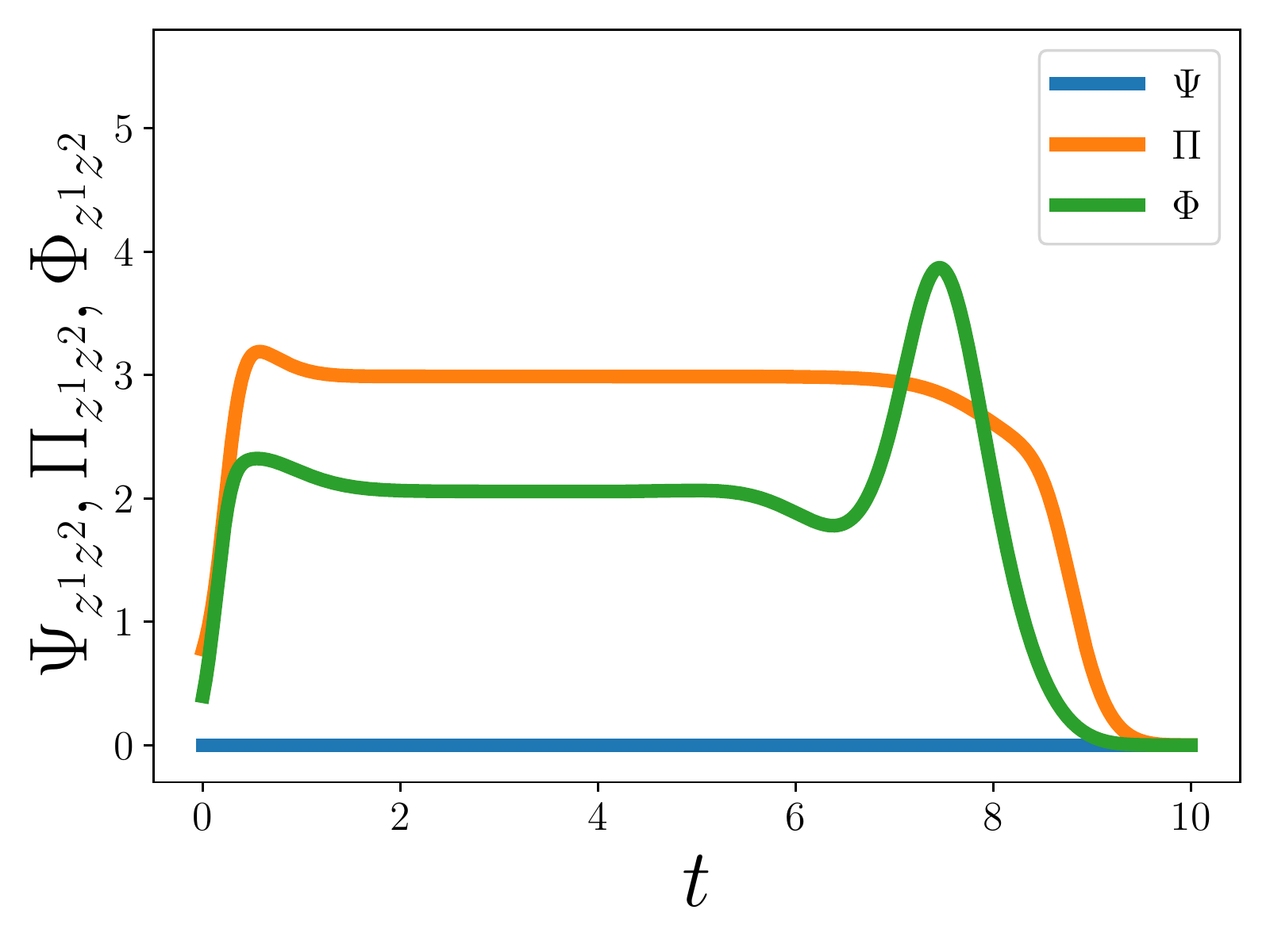}
	\end{minipage}
	\begin{minipage}[t][][b]{42mm}
		\includegraphics[width=42mm]{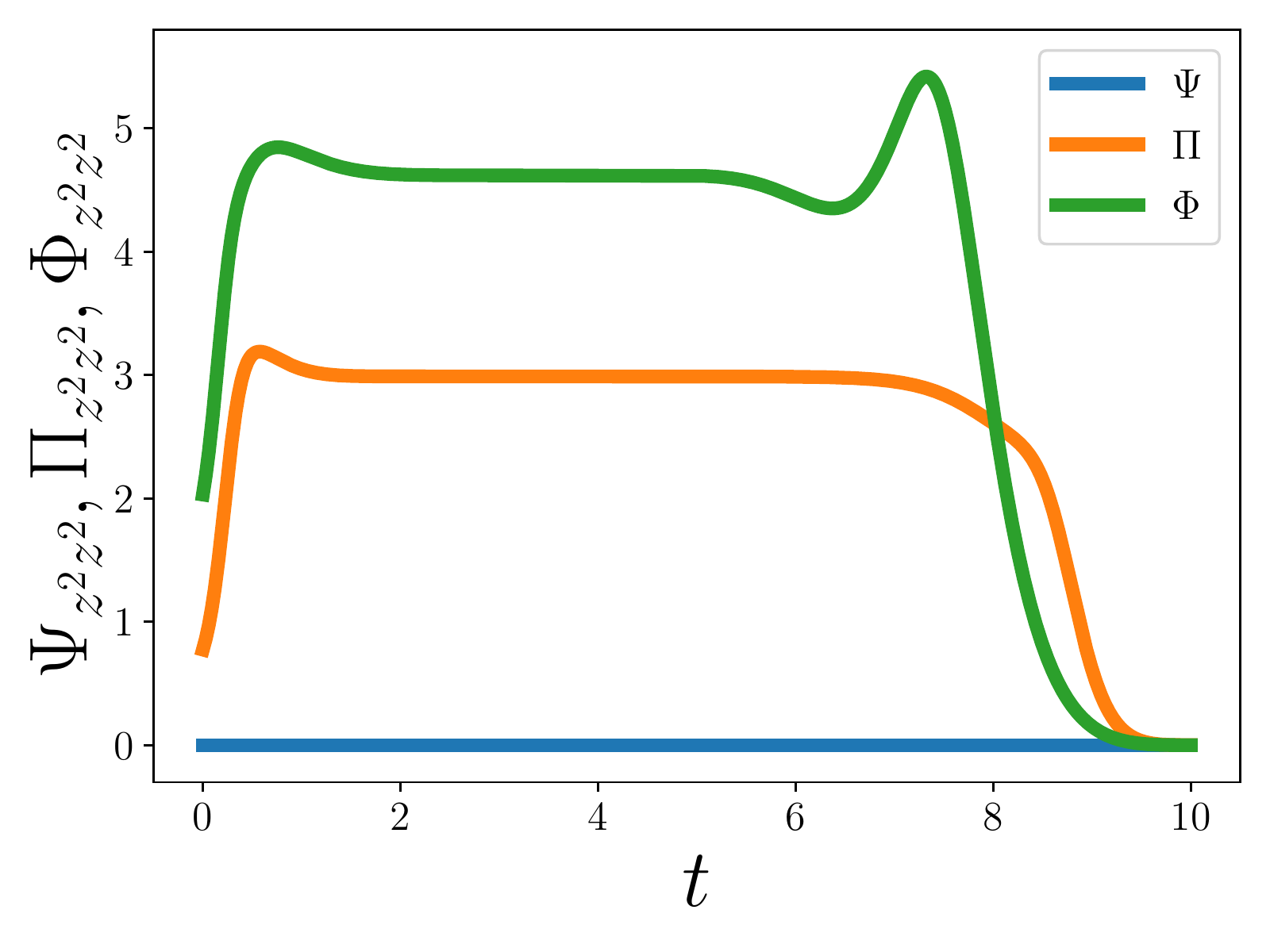}
	\end{minipage}\\
	\caption{
	Trajectories of $\Psi(t)\in\mb{R}^{3\times3}$ (blue), $\Pi(t)\in\mb{R}^{3\times3}$ (orange), and $\Phi(t)\in\mb{R}^{3\times3}$ (green), which are the solutions of (\ref{eq: ODE of Psi}), (\ref{eq: ODE of Phi}), and (\ref{eq: ODE of Pi}), respectively. 
	$\Psi(t)$, $\Pi(t)$, and $\Phi(t)$ are the optimal control gains of COSC, ML-POSC, and ML-DSC, respectively. 
	(a-f) are the elements of  $\Psi(t)$, $\Pi(t)$, and $\Phi(t)$. 
	We note that $\Psi(t)$, $\Pi(t)$, and $\Phi(t)$ are symmetric matrices. 
	}
	\label{fig: LQG ODE}
\end{center}
\end{figure*}

In this subsection, we show the significance of decentralized Riccati equation (\ref{eq: ODE of Pi}) by a numerical experiment. 
We consider the state $x_{t}\in\mb{R}$, the observation $y_{t}^{i}\in\mb{R}$ and the memory $z_{t}^{i}\in\mb{R}$ of the controller $i\in\{1,2\}$, which evolve by the following SDEs: 
\begin{align}
	dx_{t}&=\left(x_{t}+u_{t}^{1}+u_{t}^{2}\right)dt+d\omega_{t},\label{eq: state SDE LQG NE}\\
	dy_{t}^{1}&=\left(x_{t}+c_{t}^{2}\right)dt+d\nu_{t}^{1},\label{eq: observation SDE LQG NE1}\\
	dy_{t}^{2}&=\left(x_{t}+c_{t}^{1}\right)dt+d\nu_{t}^{2},\label{eq: observation SDE LQG NE2}\\
	dz_{t}^{1}&=v_{t}^{1}dt+dy_{t}^{1},\label{eq: memory SDE LQG NE1}\\
	dz_{t}^{2}&=v_{t}^{2}dt+dy_{t}^{2},\label{eq: memory SDE LQG NE2}
\end{align}
where the initial conditions are given by the standard Gaussian distributions, $\tilde{\omega}_{t}:=(\omega_{t},\nu_{t}^{1},\nu_{t}^{2})\in\mb{R}^{3}$ is the standard Wiener process, $\tilde{u}_{t}^{1}:=(u_{t}^{1},v_{t}^{1},c_{t}^{1})=\tilde{u}^{1}(t,z_{t}^{1})\in\mb{R}^{3}$ is the control of the controller 1, and $\tilde{u}_{t}^{2}:=(u_{t}^{2},c_{t}^{2},v_{t}^{2})=\tilde{u}^{2}(t,z_{t}^{2})\in\mb{R}^{3}$ is the control of the controller 2. 
Each controller can control the other controller's memory through $c_{t}^{i}$, which can be interpreted as the communication. 
The objective function to be minimized is given as follows: 
\begin{align}	
	J[\tilde{u}]:=\mb{E}_{\tilde{u}}\left[\int_{0}^{10}\left(x_{t}^{2}+(\tilde{u}_{t}^{1})^{T}\tilde{u}_{t}^{1}+(\tilde{u}_{t}^{2})^{T}\tilde{u}_{t}^{2}\right)dt\right].
	\label{eq: OF of ML-POSC LQG NE}
\end{align} 
Therefore, the objective of this problem is to minimize the state variance by the small controls. 

This problem corresponds to the LQG problem defined by (\ref{SDE of LQG}) and (\ref{OF of LQG}). 
From $s_{t}:=(x_{t},z_{t}^{1},z_{t}^{2})\in\mb{R}^{3}$, the SDEs (\ref{eq: state SDE LQG NE})--(\ref{eq: memory SDE LQG NE2}) can be rewritten as follows: 
\begin{align}
	&ds_{t}=\left(\left(\begin{array}{ccc}
		1&0&0\\
		1&0&0\\
		1&0&0\\
	\end{array}\right)s_{t}
	+\sum_{i=1}^{2}\tilde{u}_{t}^{i}\right)dt
	+d\tilde{\omega}_{t}, \nonumber
\end{align}
which corresponds to (\ref{SDE of LQG}). 
The objective function (\ref{eq: OF of ML-POSC LQG NE}) can be rewritten as follows: 
\begin{align}	
	J:=\mb{E}_{\tilde{u}}\left[\int_{0}^{10}\left(
	s_{t}^{T}\left(\begin{array}{ccc}
		1&0&0\\
		0&0&0\\
		0&0&0\\
	\end{array}\right)s_{t}
	+\sum_{i=1}^{2}(\tilde{u}_{t}^{i})^{T}\tilde{u}_{t}^{i}\right)dt\right],\nonumber
\end{align} 
which corresponds to (\ref{OF of LQG}). 
In addition, it satisfies the assumption of $R(t)$ (\ref{block diagonal assumption}). 

The Riccati (\ref{eq: ODE of Psi}) can be solved backward from the terminal condition. 
The partially observable Riccati equation (\ref{eq: ODE of Phi}) and the decentralized Riccati equation (\ref{eq: ODE of Pi}) can be solved by the forward-backward sweep method (fixed-point iteration method) \cite{lauriere_numerical_2021,tottori_notitle_2022}. 

Fig. \ref{fig: LQG ODE} shows the trajectories of $\Psi(t)$, $\Pi(t)$, and $\Phi(t)$, which are the optimal control gains of COSC, ML-POSC, and ML-DSC, respectively. 
While the memory controls do not appear in COSC, they appear in ML-POSC and ML-DSC (Fig. \ref{fig: LQG ODE}(b-f)), which indicates that the memory controls play an important role in estimation. 

We investigate $\Phi$ by comparing it with $\Psi$. 
$\Phi_{xx}$ and $\Phi_{z^{i}z^{i}}$ are larger than $\Psi_{xx}$ and $\Psi_{z^{i}z^{i}}$ (Fig. \ref{fig: LQG ODE}(a,d,f)), which may decrease $\Sigma_{xx}$ and $\Sigma_{z^{i}z^{i}}$. 
Moreover, $\Phi_{xz^{i}}$ is smaller than $\Psi_{xz^{i}}$ (Fig. \ref{fig: LQG ODE}(b,c)), which may strengthen the positive correlation between $x$ and $z^{i}$. 
Therefore, $\Phi_{xx}$, $\Phi_{z^{i}z^{i}}$, and $\Phi_{xz^{i}}$ may improve estimation, which is consistent with our discussion. 
However, $\Phi_{z^{1}z^{2}}$ is larger than $\Psi_{z^{1}z^{2}}$ (Fig. \ref{fig: LQG ODE}(e)), which may weaken the positive correlation between $z^{1}$ and $z^{2}$. 
It seems to be contrary to our discussion because it may worsen estimation. 

We compare $\Phi$ with $\Pi$ to investigate $\Phi_{z^{1}z^{2}}$. 
The absolute values of $\Phi$ are larger than those of $\Pi$ except for $\Phi_{z^{1}z^{2}}$ (Fig. \ref{fig: LQG ODE}(a,b,c,d,f)). 
This is reasonable because estimation is more important in ML-DSC than in ML-POSC. 
The problem is only $\Phi_{z^{1}z^{2}}$ (Fig. \ref{fig: LQG ODE}(e)). 
In ML-POSC, because the estimation between $z^{1}$ and $z^{2}$ is not important, $\Pi_{z^{1}z^{2}}$ is determined only from the control perspective. 
$\Pi_{z^{1}z^{2}}$ is almost the same with $\Pi_{z^{i}z^{i}}$ (Fig. \ref{fig: LQG ODE}(d,e,f)) because borrowing control is more efficient than increasing control. 
In contrast, because the estimation between $z^{1}$ and $z^{2}$ is important in ML-DSC, $\Phi_{z^{1}z^{2}}$ is smaller than $\Pi_{z^{1}z^{2}}$ (Fig. \ref{fig: LQG ODE}(e)), which may strengthen the positive correlation between $z^{1}$ and $z^{2}$. 
Therefore, $\Phi_{z^{1}z^{2}}$ is determined by a trade-off between control and estimation. 

In order to clarify the significance of the decentralized Riccati equation (\ref{eq: ODE of Pi}), 
we compare the performance of the optimal control function (\ref{eq: optimal control of LQG}) with that of the following control functions: 
\begin{align}
	&u^{i,\Psi}(t,z^{i})=-R_{ii}^{-1}B_{i}^{T} \left(\Psi K_{i}\hat{s}+\Psi\mu\right), \label{eq: Psi control of LQG}\\
	&u^{i,\Pi}(t,z^{i})=-R_{ii}^{-1}B_{i}^{T} \left(\Pi K_{i}\hat{s}+\Psi\mu\right), \label{eq: Phi control of LQG}
\end{align}
which replaces $\Phi$ with $\Psi$ and $\Pi$, respectively. 
We note that the second terms are not important because $\mu(t)=0$ in this problem. 
The result is shown in Fig. \ref{fig: LQG SDE}.  
The expected cumulative cost of (\ref{eq: Psi control of LQG}) is larger than that of (\ref{eq: optimal control of LQG}) (Fig. \ref{fig: LQG SDE}(d)) because (\ref{eq: Psi control of LQG}) does not account the estimation of the state and the other memory.  
Moreover, the expected cumulative cost of (\ref{eq: Phi control of LQG}) is larger than that of (\ref{eq: optimal control of LQG}) (Fig. \ref{fig: LQG SDE}(d)) because (\ref{eq: Phi control of LQG}) does not account the estimation of the other memory. 
These results indicate that the decentralized Riccati equation (\ref{eq: ODE of Pi}) is significant in ML-DSC. 

\begin{figure*}[t]
\begin{center}
	\begin{minipage}[t][][b]{42mm}
	(a)
	\end{minipage}
	\begin{minipage}[t][][b]{42mm}
	(b)
	\end{minipage}
	\begin{minipage}[t][][b]{42mm}
	(c)
	\end{minipage}
	\begin{minipage}[t][][b]{42mm}
	(d)
	\end{minipage}\\
	\begin{minipage}[t][][b]{42mm}
		\includegraphics[width=42mm]{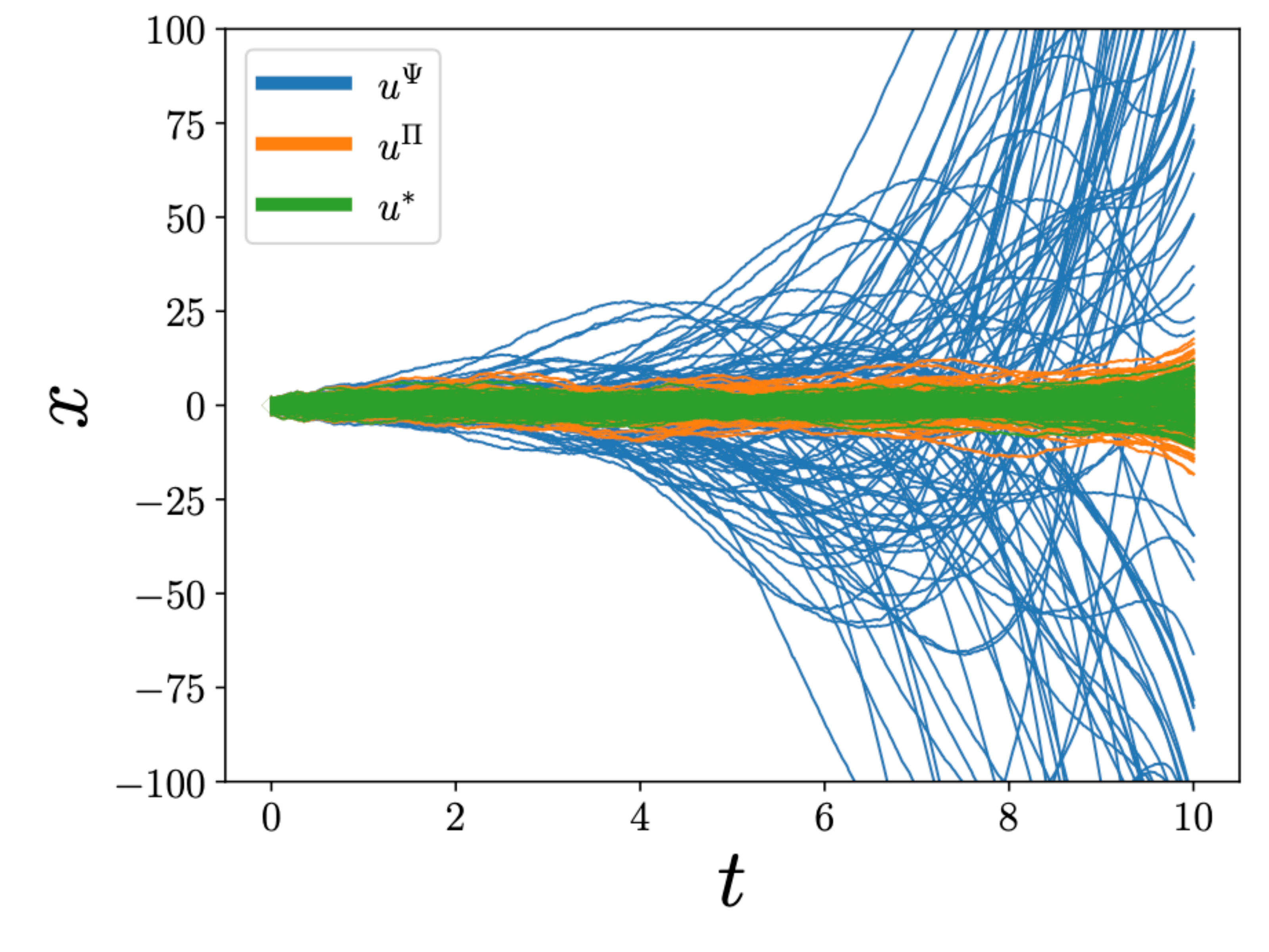}
	\end{minipage}
	\begin{minipage}[t][][b]{42mm}
		\includegraphics[width=42mm]{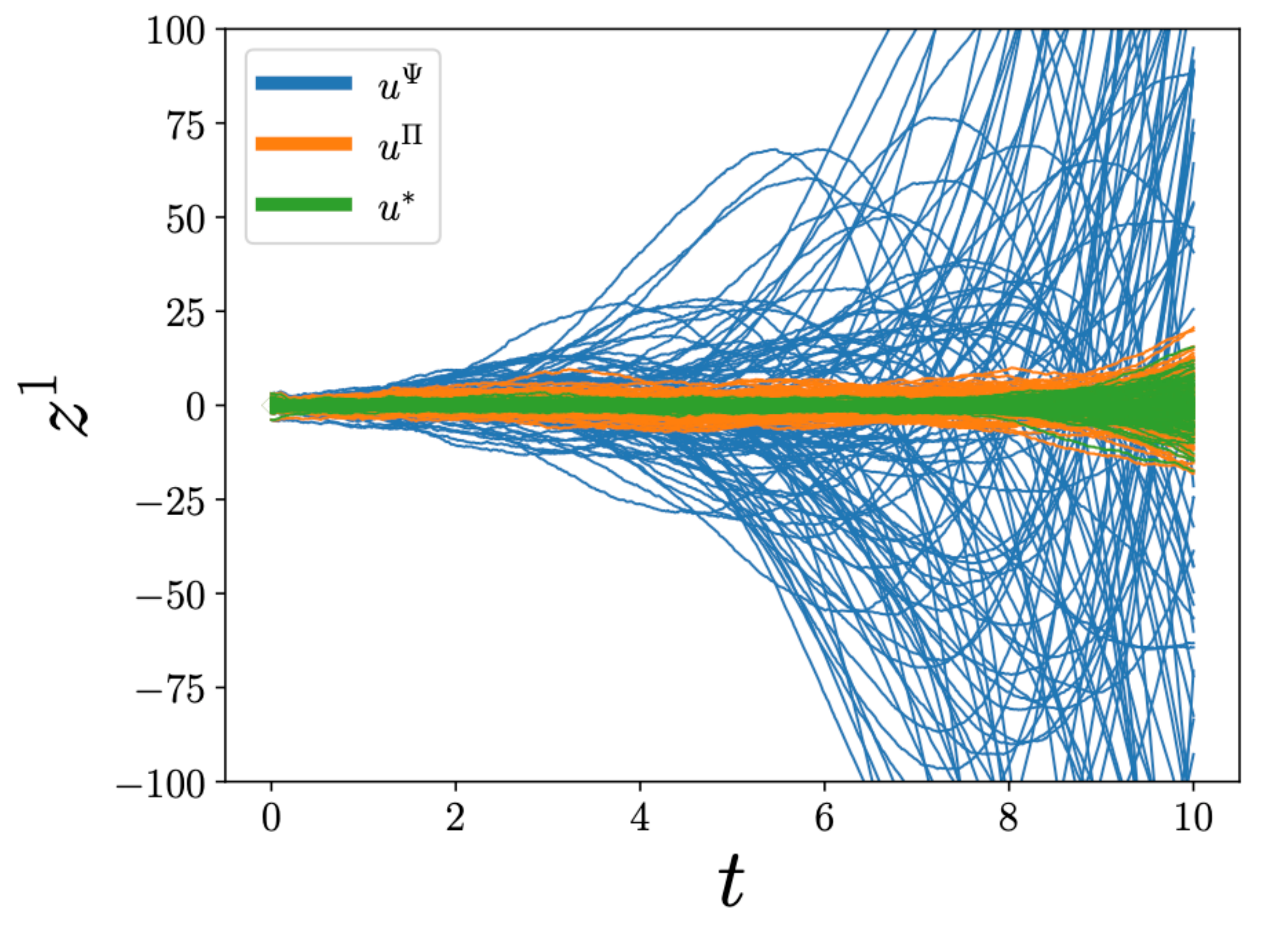}
	\end{minipage}
	\begin{minipage}[t][][b]{42mm}
		\includegraphics[width=42mm]{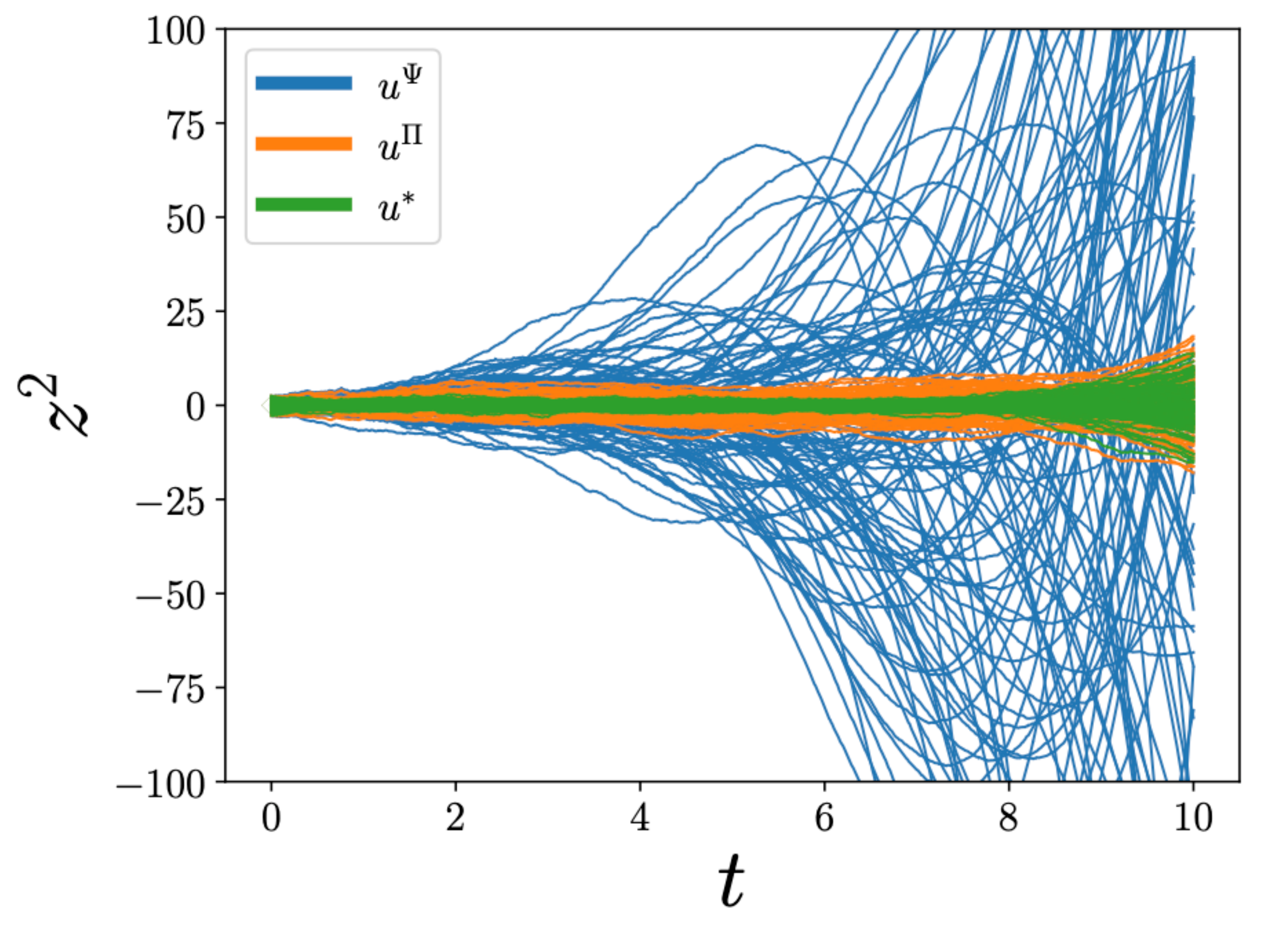}
	\end{minipage}
	\begin{minipage}[t][][b]{42mm}
		\includegraphics[width=42mm]{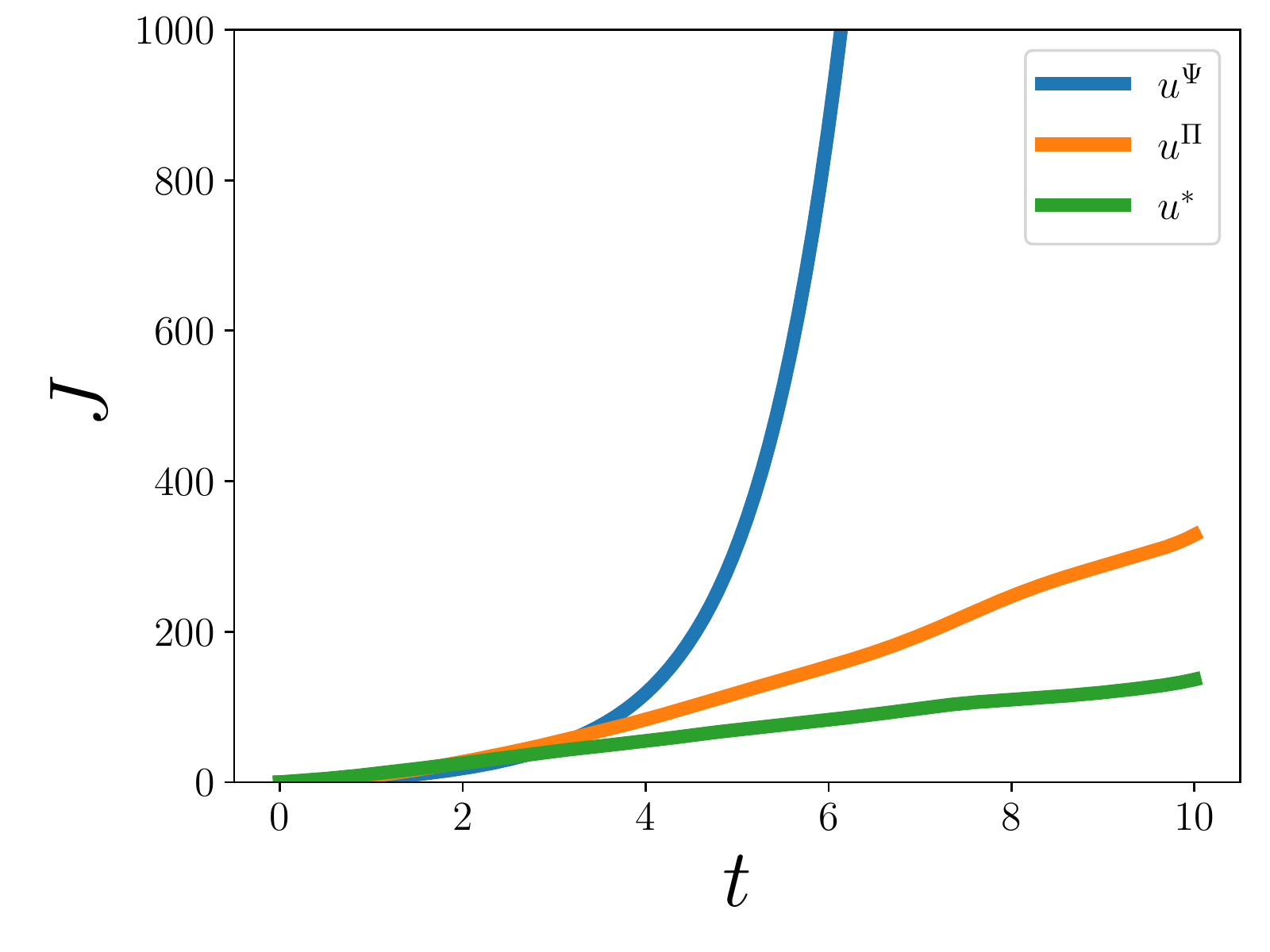}
	\end{minipage}\\
	\caption{
	(a,b,c) Stochastic behaviors of the state $x_{t}$ (a), the controller 1's memory $z_{t}^{1}$ (b), and the controller 2's memory $z_{t}^{2}$ (c) for 100 samples. 
	(d) The expected cumulative cost calculated from 100 samples. 
	Blue, orange, and green curves are controlled by $u^{\Psi}$ (\ref{eq: Psi control of LQG}), $u^{\Pi}$ (\ref{eq: Phi control of LQG}), and $u^{*}$ (\ref{eq: optimal control of LQG}), respectively. 
	}
	\label{fig: LQG SDE}
\end{center}
\end{figure*}

\section{CONCLUSION}\label{sec: Conclusion}
In this work, we proposed ML-DSC, in which each controller compresses the observation history into the finite-dimensional memory. 
Because this compression simplifies the estimation among the controllers, ML-DSC can be solved in a general case based on the mean-field control theory. 
We demonstrated ML-DSC in the general LQG problem involving a non-nested structure. 
Because estimation and control are not clearly separated in the general LQG problem, the Riccati equation is modified to the decentralized Riccati equation, which may improve estimation as well as control. 
Our numerical experiment showed that the decentralized Riccati equation is superior to the conventional Riccati equations. 

ML-DSC can be solved in practice even in a non-LQG problem. 
The optimal control function of ML-DSC is obtained by solving the system of HJB-FP equations. 
Because the system of HJB-FP equations also appears in the mean-field game and control, numerous numerical algorithms have been developed \cite{lauriere_numerical_2021}. 
Especially, neural network-based algorithms have been proposed recently, which can solve high-dimensional state problems efficiently \cite{ruthotto_machine_2020,lin_alternating_2021}. 
By exploiting these algorithms, we may efficiently solve ML-DSC consisting of a large number of agents.

\section*{APPENDIX}
\subsection{Proof of Theorem \ref{theo: optimal control of GML-DSC based on Bellman eq}}
We define the value function $V(t,p)$ as follows: 
\begin{align}
	V(t,p):=\min_{u_{t:T}}\left[\int_{t}^{T}\bar{f}(\tau,p_{\tau},u_{\tau})d\tau+\bar{g}_{T}(p_{T})\right],
\end{align}
where $\{p_{\tau}|\tau\in[t,T]\}$ is the solution of the FP equation (\ref{eq: FP eq}) where $p_{t}=p$. 
From the simple calculation \cite{tottori_mean-field_2022}, the following Bellman equation is obtained: 
\begin{align}
	-\frac{\pl V(t,p)}{\pl t}=\min_{u}\mb{E}_{p(s)}\left[H\left(t,s,u,\frac{\delta V(t,p)}{\delta p}(s)\right)\right]. \nonumber
\end{align}
Minimizing the right-hand side with respect to $u$ except for $u^{i}$, the following equation is obtained: 
\begin{align}
	-\frac{\pl V(t,p)}{\pl t}=\min_{u^{i}}\mb{E}_{p(s)}\left[H\left(t,s,(u^{-i*},u^{i}),\frac{\delta V(t,p)}{\delta p}(s)\right)\right]. \nonumber
\end{align}
Because the control $u^{i}$ is the function of the memory $z^{i}$ in ML-DSC, 
the minimization by $u^{i}$ can be exchanged with the expectation by $p(z^{i})$ as follows: 
\begin{align}
	-\frac{\pl V(t,p)}{\pl t}=\mb{E}_{p(z^{i})}\left[\min_{u^{i}}\mb{E}_{p(s^{-i}|z^{i})}\left[H\right]\right].\nonumber
\end{align}
From the optimal control theory \cite{yong_stochastic_1999}, the optimal control function is given by the right-hand side of the Bellman equation as follows:  
\begin{align}
	u^{i*}(t,z^{i},p)=\argmin_{u^{i}}\mb{E}_{p(s^{-i}|z^{i})}\left[H\right].\nonumber
\end{align}
Because the FP equation (\ref{eq: FP eq}) is deterministic, the optimal control function is given by  $u^{i*}(t,z^{i})=u^{i*}(t,z^{i},p_{t})$. 

\subsection{Proof of Theorem \ref{theo: optimal control of LQG in ML-DSC}}
In the LQG problem, the Hamiltonian is given by
\begin{align}
	&H(t,s,u,w)
	=s^{T}Qs+\sum_{i=1}^{N}(u^{i})^{T}R_{ii}u^{i}
	+\left(\frac{\pl w(t,s)}{\pl s}\right)^{T}As\nonumber\\
	&+\left(\frac{\pl w(t,s)}{\pl s}\right)^{T}\sum_{i=1}^{N}B_{i}u^{i}
	+\frac{1}{2}\tr\left\{\frac{\pl}{\pl s}\left(\frac{\pl w(t,s)}{\pl s}\right)^{T}\sigma\sigma^{T}\right\}.\nonumber
\end{align}
From Theorem \ref{theo: optimal control of GML-DSC} and the stationary condition, the optimal control function is given by 
\begin{align}
	u^{i*}(t,z^{i})
	&=-\frac{1}{2}R_{ii}^{-1}B_{i}^{T}\mb{E}_{p_{t}(s^{-i}|z^{i})}\left[\frac{\pl w(t,s)}{\pl s}\right].
	\label{eq: optimal control of LQG ver1}
\end{align}

We assume that $p_{t}(s)$ is the Gaussian distribution
\begin{align}
	p_{t}(s):=\mcal{N}\left(s|\mu(t),\Sigma(t)\right), 
	\label{eq: FP assumption LQG}
\end{align}
and $w(t,s)$ is the quadratic function
\begin{align}
	w(t,s)=s^{T}\Phi(t)s+\alpha^{T}(t)s+\beta(t). 
	\label{eq: HJB assumption LQG}
\end{align}
In this case, the optimal control function (\ref{eq: optimal control of LQG ver1}) can be calculated as follows: 
\begin{align}
	u^{i*}(t,z^{i})=-\frac{1}{2}R_{ii}^{-1}B_{i}^{T}\left(2\Phi K_{i}\hat{s}+2\Phi\mu+\alpha\right).
	\label{eq: optimal control of LQG ver3}
\end{align}
Because (\ref{eq: optimal control of LQG ver3}) is linear with respect to $\hat{s}$, 
$p_{t}(s)$ becomes the Gaussian distribution, which is consistent with our assumption (\ref{eq: FP assumption LQG}). 

Substituting (\ref{eq: HJB assumption LQG}) and (\ref{eq: optimal control of LQG ver3}) into the HJB equation (\ref{eq: HJB eq}), 
the following ordinary differential equations are obtained: 
\begin{align}
	-\dot{\Phi}&=Q+A^{T}\Phi+\Phi A-\Phi BR^{-1}B^{T}\Phi+\mcal{Q},\label{eq: ODE of Pi pre}\\
	-\dot{\alpha}&=(A-BR^{-1}B^{T}\Phi)^{T}\alpha-2\mcal{Q}\mu,\label{eq: ODE of alpha}\\
	-\dot{\beta}&=\tr(\Phi\sigma\sigma^{T})-\frac{1}{4}\alpha^{T}BR^{-1}B^{T}\alpha+\mu^{T}\mcal{Q}\mu,\label{eq: ODE of beta}
\end{align}
where $\mcal{Q}:=\sum_{i=1}^{N}(I-K_{i})^{T}\Phi B_{i}R_{ii}^{-1}B_{i}^{T}\Phi (I-K_{i})$, $\Psi(T)=O$, $\alpha(T)=0$, and $\beta(T)=0$. 
If $\Phi(t)$, $\alpha(t)$, and $\beta(t)$ satisfy (\ref{eq: ODE of Pi pre}), (\ref{eq: ODE of alpha}), and (\ref{eq: ODE of beta}), respectively, the HJB equation (\ref{eq: HJB eq}) is satisfied, which is consistent with our assumption (\ref{eq: HJB assumption LQG}). 

Defining $\Upsilon(t)$ by $\alpha(t)=2\Upsilon(t)\mu(t)$ and $\Psi(t)$ by $\Psi(t):=\Phi(t)+\Upsilon(t)$, the optimal control function (\ref{eq: optimal control of LQG ver3}) can be calculated as follows: 
\begin{align}
	u^{i*}(t,z^{i})=-R_{ii}^{-1}B_{i}^{T}\left(\Phi K_{i} \hat{s}+\Psi\mu\right). 
\end{align}
From (\ref{eq: ODE of Pi pre}) and (\ref{eq: ODE of alpha}), $\Psi(t)$ is the solution of the Riccati equation (\ref{eq: ODE of Psi}) \cite{tottori_mean-field_2022}. 


\bibliographystyle{ieeetr}
\bibliography{220531_ML-DSC_ref}

\end{document}